\documentclass[journal]{IEEEtran}
\usepackage{algcompatible}
\usepackage{algpseudocode}
\usepackage{multirow}
\usepackage{multirow}
\usepackage{xspace}
\usepackage{hyperref}
\usepackage{listings}
\usepackage{url}
\usepackage{csquotes}
\usepackage[mathscr]{euscript} 
\usepackage{amsfonts,epsfig}
\usepackage{enumerate}
\usepackage{amsmath,amscd,amsfonts,amssymb}
\newcommand{\MATLAB}{\textsc{Matlab}\xspace}

\usepackage{lineno,hyperref}
\usepackage{graphicx,subfigure}
\usepackage{amsfonts,epsfig}

\usepackage[mathscr]{euscript} 
\usepackage{enumerate}
\usepackage{amsmath,amscd,amsfonts,amssymb}
\usepackage{times}
\usepackage{extarrows}
\usepackage{stmaryrd}
\usepackage{textcomp}
\usepackage[mathscr]{euscript}
\usepackage{algorithm2e}
\usepackage{colortbl}
\usepackage{color,soul}
\usepackage{newunicodechar}
\usepackage{amsthm}
\theoremstyle{definition}
\newtheorem{thm}{Theorem}

\newtheorem{lem}[thm]{Lemma}
\theoremstyle{definition}%
\newtheorem{defn}{Definition}%
\newtheorem{exa}{Example}

\newcommand{\C}{\underline{\bf{C}}}
\newcommand{\A}{\underline{\bf A}}

\newcommand{\X}{\underline{\bf X}}
\newcommand{\Y}{\underline{\bf Y}}
\newcommand{\I}{\underline{\bf I}}
\newcommand{\B}{\underline{\bf B}}
\newcommand{\Q}{\underline{\bf Q}}
\newcommand{\U}{\underline{\bf U}}
\newcommand{\tS}{\underline{\bf S}}
\newcommand{\V}{\underline{\bf V}}
\usepackage{lineno,hyperref}
\usepackage{amsmath, amscd,amsfonts, amssymb }
\usepackage{algorithm2e}
\ifCLASSINFOpdf
\else
   \usepackage[dvips]{graphicx}
\fi
\usepackage{url}
\usepackage{graphicx}
\ifCLASSINFOpdf
\else
\fi

\begin{document}

\title{An Efficient Randomized Fixed-Precision Algorithm for Tensor Singular Value Decomposition}

 \author{
 Salman~Ahmadi-Asl$^{\dag}$

\thanks{$^{\dag}$Skolkovo Institute of Science and Technology (SKOLTECH), Center for Artificial Intelligence Technology, Moscow, E-mail: s.asl@skoltech.ru. 
}
}

\maketitle

\begin{abstract}
The existing randomized algorithms need an initial estimation of the tubal rank to compute a tensor singular value decomposition. This paper proposes a new randomized fixed-precision algorithm which for a given third-order tensor and a prescribed approximation error bound, automatically finds an optimal tubal rank and the corresponding low tubal rank approximation. The algorithm is based on the random projection technique and equipped with the power iteration method for achieving a better accuracy. We conduct simulations on synthetic and real-world datasets to show the efficiency and performance of the proposed algorithm. 
\end{abstract}


\IEEEpeerreviewmaketitle

\section{Introduction}

\IEEEPARstart{T}{ensor} decompositions are efficient tools for multi-way data processing (analysis). In particular, they can be used for making reduction on the data tensors and compressing them without destroying their intrinsic multidimensional structure. In the past few decades, several types of tensor decompositions have been introduced such as CANDECOMP/PARAFAC decomposition (CPD) \cite{hitchcock1928multiple,hitchcock1927expression}, Tucker decomposition \cite{tucker1963implications,tucker1964extension,tucker1966some} and its special case, i.e. Higher Order SVD \cite{de2000multilinear}, Block Term decomposition \cite{de2008decompositionsI,de2008decompositionsII,de2008decompositionsIII}, Tensor Train/Tensor Ring\footnote{It is also known as Tensor Chain decomposition.} decomposition \cite{oseledets2011tensor,zhao2016tensor,espig2012note}, tensor SVD (t-SVD) \cite{kilmer2011factorization,kilmer2013third,braman2010third}, each of which generalizes the notion of matrix rank to tensors in an efficient way. They have been successfully applied in many applications such as signal processing \cite{cichocki2015tensor}, machine learning \cite{sidiropoulos2017tensor}, blind source separation (\cite{comon2010handbook,cichocki2015tensor}), chemometrics, \cite{leardi2005multi}, feature extraction/selection \cite{phan2010tensor}. For a comprehensive study on tensors and their applications, we refer to (\cite{cichocki2016tensor,cichocki2017tensor,kolda2009tensor}). Tensor SVD (t-SVD) is a special kind of tensor decomposition representing a third-order tensor as the t-product of three third-order tensors where, the middle tensor is f-diagonal  (\cite{kilmer2011factorization,kilmer2013third,braman2010third,gleich2013power}).
This decomposition has found many applications including deep learning (\cite{newman2018stable,newman2019step}), tensor completion (\cite{zhang2016exact,zhang2014novel}), numerical analysis (\cite{lund2018new,lund2020tensor,miao2020generalized,miao2020t}), image reconstruction \cite{soltani2016tensor}. There are mainly several randomized algorithms (\cite{zhang2018randomized,tarzanagh2018fast,qi2021t}) to decompose a tensor into the t-SVD format but all of them need an estimation of the tubal rank\footnote{Such randomized algorithms are called randomized fixed-rank algorithms.} which may not be an obvious task. The work \cite{tarzanagh2018fast} utilizes the randomized sampling framework while the random projection method is exploited in \cite{zhang2018randomized,qi2021t}. Motivated by the limitation of initial tubal rank estimation, in this paper, we propose a new randomized fixed-precision algorithm which is independent of initial tubal rank estimation. More precisely, for a given third-order tensor and an approximation error 
bound, it retrieves the optimal tubal rank and corresponding low tubal rank approximation automatically. Our algorithm is a generalization of those proposed in (\cite{yu2018efficient, martinsson2016randomized}) from matrices to tensors. To the best of our knowledge, this is the first fixed-precision algorithm proposed for the t-SVD\footnote{After acceptance of the paper, the author found similar incremental algorithms for computation of the t-SVD in \cite{ugwu2021tensor,ugwu2021viterative}.}. To verify the efficiency of the proposed algorithm, we apply it to both synthetic and real-world datasets. 

The organization of this paper is structured as follows: The preliminary concepts and definitions used throughout the paper are introduced in Section \ref{Pre}. Section \ref{Sec:TSVD} is devoted to introducing the tensor SVD, tensor QR (t-QR) and algorithms to compute them. The proposed algorithm is presented in Section \ref{ProApp} with a detailed discussion on its properties. The computational complexity analysis is presented in Section \ref{sec:com}. Simulation results are presented in Section \ref{Sec:Sim} and a conclusion is given in Section \ref{Sec:Conclu}.

\section{Preliminaries}\label{Pre}
In this section, we present basic notations and concepts which we need throughout the paper. For simplicity of presentation, we have summarized all notations in Table \ref{Tab_no}. Tensors, matrices and vectors are denoted by underlined bold upper case letters e.g., $\underline{\bf X}$, bold upper case letters, e.g., ${\bf X}$ and bold lower case letters, e.g., ${\bf x}$, respectively. 
Slices are obtained by fixing all but two modes of a tensor. For a third-order tensor $\X$, the slices ${\X}(:,:,k),\,{\X}(:,j,:)$ and ${\X}(i,:,:)$ are called frontal, lateral and horizontal slices. Fibers are obtained by fixing all but one mode of a tensor. For a third-order tensor $\X$, ${\X}(i,j,:)$ is called a tube. The Frobenius norm of tensors is denoted by ${\left\| . \right\|_F}$. The notation ``${\rm conj}$'' denotes the complex conjugate of a complex number or the component-wise complex conjugate of a matrix. $\lceil n\rceil$ means the nearest integer number greater than or equal to $n$. Throughout the paper, we consider real tensors, however, our results can be easily generalized to complex tensors. 

One of the basic operators used in our work is concatenation along the first and second modes. A concatenation along mode 1 of tensors $\underline{\bf A}\in\mathbb{R}^{I_1\times I_2\times I_3}$ 
and $\underline{\bf B}\in\mathbb{R}^{J_1\times J_2\times J_3}$  is denoted as $\underline{\bf C} = \underline{\bf A} \boxplus_1 \underline{\bf B} \in {\mathbb{R}^{(I_1+J_1)\times I_2\times I_3}},$  where $I_2=J_2,\,I_3=J_3$ (see Figure \ref{fig1}, (a)) 
and the same definition can be stated for concatenation along the second mode\footnote{In \MATLAB the command "{\sffamily cat}" can be used for the concatenation operation.} (see Figure \ref{fig1}, (b)) \cite{cichocki2016tensor}. In the paper, sometime we use alternative notations $\A\boxplus_1 \B \equiv \begin{bmatrix}
\A\\ \B
\end{bmatrix}$ and $\A\boxplus_2 \B \equiv [\A,\B],$ for a better understanding of the developed algorithm.
\begin{figure}
\begin{center}\label{concate}
\includegraphics[width=6 cm,height=4 cm]{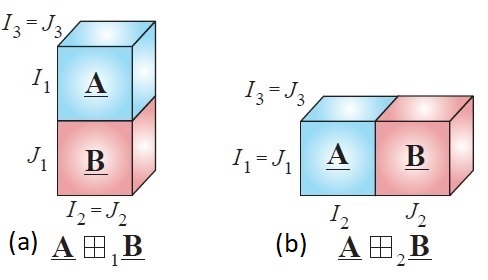}\\
\caption{\small{Concatenation of tensors along the first mode (a) and the second mode (b) \cite{cichocki2016tensor}.}}\label{fig1}
\end{center}
\end{figure}

\begin{table}
\begin{center}
\caption{Notation description}\label{Tab_no}
\begin{tabular}{ c|c  }
 Notation & Description  \\\hline 
 $\underline{\bf X},\,{\bf X},\,{\bf x}$ & A tensor, a matrix and a vector  \\\hline  
 $\|\underline{\bf X}\|_F$ & Frobenius norm of a tensor\\\hline
 $\langle{\X},{\Y}\rangle$ & Inner product of two tensors\\\hline
 $\X*\Y$ & t-product of two tensors\\\hline
 $\I$& Identity tensor \\\hline
 $\X^T$ & Transpose of a tensor\\\hline
 $\X\boxplus_1\Y$ & Concatenation of tensors along the first mode\\\hline
 $\X\boxplus_2\Y$ & Concatenation of tensors along the second mode\\\hline
 {\rm conj}({\bf X}) & Component-wise complex conjugate 
\end{tabular}
\end{center}
\end{table}

{
\begin{defn}
(Inner product) The inner product of two tensors $\underline{\bf X}\in\mathbb{R}^{I_1\times I_2\times I_3}$ and $\underline{\bf Y}\in\mathbb{R}^{I_1\times I_2\times I_3}$ is defined as 
\[
\langle
{\underline{\bf X}},{\underline{\bf Y}}
\rangle
=\sum_{i_1=1}^{I_1}
 \sum_{i_2=1}^{I_2}\sum_{i_3=1}^{I_3}
 x_{i_1i_2i_3}y_{i_1i_2i_3}.
\] It can be shown that $\langle\underline{\bf X},\underline{\bf X}\rangle=\|\X\|^2_F$.
\end{defn}
}
In order to introduce the t-SVD, we first need to present some basic definitions such as the t-product operation and f-diagonal tensors.

\begin{defn} ({t-product})
Let $\underline{\mathbf X}\in\mathbb{R}^{I_1\times I_2\times I_3}$ and $\underline{\mathbf Y}\in\mathbb{R}^{I_2\times I_4\times I_3}$, the t-product $\underline{\mathbf X}*\underline{\mathbf Y}\in\mathbb{R}^{I_1\times I_4\times I_3}$ is defined as follows
\begin{equation}\label{TPROD}
\underline{\mathbf C} = \underline{\mathbf X} * \underline{\mathbf Y} = {\rm fold}\left( {{\rm circ}\left( \underline{\mathbf X} \right){\rm unfold}\left( \underline{\mathbf Y} \right)} \right),
\end{equation}
where 
\[
{\rm circ} \left(\underline{\mathbf X}\right)
=
\begin{bmatrix}
\underline{\mathbf X}(:,:,1) & \underline{\mathbf X}(:,:,I_3) & \cdots & \underline{\mathbf X}(:,:,2)\\
\underline{\mathbf X}(:,:,2) & \underline{\mathbf X}(:,:,1) & \cdots & \underline{\mathbf X}(:,:,3)\\
 \vdots & \vdots & \ddots &  \vdots \\
 \underline{\mathbf X}(:,:,I_3) & \underline{\mathbf X}(:,:,I_3-1) & \cdots & \underline{\mathbf X}(:,:,1)
\end{bmatrix},
\]
and 
\[
{\rm unfold}(\underline{\mathbf Y})=
\begin{bmatrix}
\underline{\mathbf Y}(:,:,1)\\
\underline{\mathbf Y}(:,:,2)\\
\vdots\\
\underline{\mathbf Y}(:,:,I_3)
\end{bmatrix},\hspace*{.5cm}
\underline{\mathbf Y}={\rm fold} \left({\rm unfold}\left(\underline{\mathbf Y}\right)\right).
\]
\end{defn}

In view of \eqref{TPROD}, it is seen that the t-product operation is the block circular convolution operator and because of this, it can be easily computed through Fast Fourier Transform (FFT). More precisely, we first transform all tubes of two tensors $\underline{\bf X}$ and $\underline{\bf Y}$, into the Fourier domain denoted by $\widehat{\X}$ and $\widehat{\Y}$, then multiply frontal slices of the tensors $\widehat{\X},\,\,\widehat{\Y}$ and obtain a new tensor $\widehat{\C}$. Finally we apply the Inverse FFT (IFFT) to all tubes of the resulting tensor $\widehat{\C}$ . This procedure is summarized in Algorithm \ref{ALG:TSVDP}. Note that the \MATLAB command ${\rm fft}\left( {\underline{\mathbf X},[],3} \right)$ is equivalent to computing the FFT of all tubes of the tensor $\underline{\bf X}$.
The t-product can be defined according to an arbitrary invertible transform \cite{kernfeld2015tensor}. For example, the unitary transform matrices are utilized in  \cite{song2020robust}, instead of the discrete Fourier transform matrices. It was also proposed in \cite{jiang2020framelet} to use  non-invertible transforms instead of unitary matrices.

It is interesting to note that due to the special structure of the discrete Fourier transforms, we can reduce the $I_3$ matrix-matrix multiplications in Algorithm \ref{ALG:TSVDP} to only $\lceil \frac{I_3+1}{2}\rceil$, (see  \cite{hao2013facial,lu2019tensor} for details). To do so, we can consider the following computations instead of the ``for'' loop in Algorithm \ref{ALG:TSVDP} 
\begin{eqnarray}
\nonumber
\widehat{\underline{\mathbf C}}\left( {:,:,i} \right) &=& \widehat{\underline{\mathbf X}}\left( {:,:,i} \right)\,\widehat{\underline{\mathbf Y}}\left( {:,:,i} \right),\\
i&=&1,2,\ldots,\lceil\label{er} \frac{I_3+1}{2}\rceil,\\
\nonumber
\widehat{\underline{\mathbf C}}\left( {:,:,i} \right)&=&{\rm conj}(\widehat{\underline{\mathbf C}}\left( {:,:,I_3-i+2} \right))\\\label{er_2}
i&=&\lceil \frac{I_3+1}{2}\rceil+1,\ldots,I_3.
\end{eqnarray}
Having computed the multiplications in \eqref{er}, the rest of computations in \eqref{er_2} are not expensive. So this version of the t-product is faster than its naive implementation and we use it in all our simulations\footnote{The \MATLAB implementation of this multiplication and related tubal operations are provided in the toolbox \url{https://github.com/canyilu/Tensor-tensor-product-toolbox}}.

\RestyleAlgo{ruled}
\LinesNumbered
\begin{algorithm}
\SetKwInOut{Input}{Input}
\SetKwInOut{Output}{Output}\Input{Two data tensors $\underline{\mathbf X} \in {\mathbb{R}^{{I_1} \times {I_2} \times {I_3}}},\,\,\underline{\mathbf Y} \in {\mathbb{R}^{{I_2} \times {I_4} \times {I_3}}}$} 
\Output{t-product $\underline{\mathbf C} = \underline{\mathbf X} * \underline{\mathbf Y}\in\mathbb{R}^{I_1\times I_4\times I_3}$}
\caption{t-product in the Fourier domain \cite{kilmer2011factorization}}\label{ALG:TSVDP}
      {
      $\widehat{\underline{\mathbf X}} = {\rm fft}\left( {\underline{\mathbf X},[],3} \right)$\\
      $\widehat{\underline{\mathbf Y}} = {\rm fft}\left( {\underline{\mathbf Y},[],3} \right)$\\
\For{$i=1,2,\ldots,I_3$}
{                        
$\widehat{\underline{\mathbf C}}\left( {:,:,i} \right) = \widehat{\underline{\mathbf X}}\left( {:,:,i} \right)\,\widehat{\underline{\mathbf Y}}\left( {:,:,i} \right)$\\
}
$\underline{\mathbf C} = {\rm ifft}\left( {\widehat{\underline{\mathbf C}},[],3} \right)$   
       	}       	
\end{algorithm}

\begin{defn} ({Transpose})
The transpose of a tensor $\underline{\mathbf X}\in\mathbb{R}^{I_1\times I_2\times I_3}$ is denoted by $\underline{\mathbf X}^{T}\in\mathbb{R}^{I_2\times I_1\times I_3}$ and is produced by applying the transpose to all frontal slices of the tensor $\underline{\mathbf X}$ and reversing the order of the transposed frontal slices 2 through $I_3$.
\end{defn}

\begin{defn} ({Identity tensor})
Identity tensor $\underline{\mathbf I}\in\mathbb{R}^{I_1\times I_1\times I_3}$ is a tensor whose first frontal slice is an identity matrix of size $I_1\times I_1$ and all other frontal slices are zero. It is easy to show $\I*\X=\X$ and $\X*\I =\X$ for all tensors of conforming sizes.
\end{defn}
\begin{defn} ({Orthogonal tensor})
A tensor $\underline{\mathbf X}\in\mathbb{R}^{I_1\times I_1\times I_3}$ is orthogonal (under t-product operator) if ${\underline{\mathbf X}^T} * \underline{\mathbf X} = \underline{\mathbf X} * {\underline{\mathbf X}^ T} = \underline{\mathbf I}$.
\end{defn}

\begin{defn} ({f-diagonal tensor})
If all frontal slices of a tensor are diagonal then the tensor is called f-diagonal.
\end{defn}

\begin{defn}
(Random tensor) A tensor $\underline{\bf \Omega}$ is random if its first frontal slice $\underline{\bf \Omega}(:,:,1)$ has independent and identically distributed (i.i.d) elements while the other frontal slices are zero.
\end{defn}

\begin{defn}\label{Def1}
Let $\mathcal{V}_1$ and $\mathcal{V}_2$ be two vector spaces with corresponding  inner products ${\left\langle . \right\rangle_{\mathcal{V}_1}}$ and ${\left\langle . \right\rangle_{\mathcal{V}_2}}$ respectively and $\mathcal{L}:{\mathcal{V}_1} \to {\mathcal{V}_2}$ be a linear operator. The adjoint of the operator $\mathcal{L}$ is denoted by ${\mathcal{L}}^{adj}$ and satisfies the following relation 
\[
\left\langle {\mathcal{L}\left( \underline{\bf X} \right),\underline{\bf Y}} \right\rangle_{\mathcal{V}_2} = \left\langle {\underline{\bf X},{\mathcal{L}^{adj}}\left( \underline{\bf Y} \right)} \right\rangle_{\mathcal{V}_1}.
\]
\end{defn}

\section{Tensor SVD (t-SVD) and Tensor QR (t-QR) decomposition}\label{Sec:TSVD}
Tensor SVD (t-SVD) is a special type of tensor decomposition representing a third-order tensor as a product of three third-order tensors where the middle tensor is f-diagonal (\cite{kilmer2011factorization,kilmer2013third,braman2010third,gleich2013power}), see Figure \ref{Pic7}, for a graphical illustration of the t-SVD and its truncated version. The generalization of the t-SVD to higher order tensors is given in \cite{martin2013order}. Throughout this paper for the t-SVD, we only focus on third-order tensors, though the extension of our approach to tensors of order higher higher than 3 is straightforward. 

\begin{figure}
\begin{center}
\includegraphics[width=9 cm,height=4.5 cm]{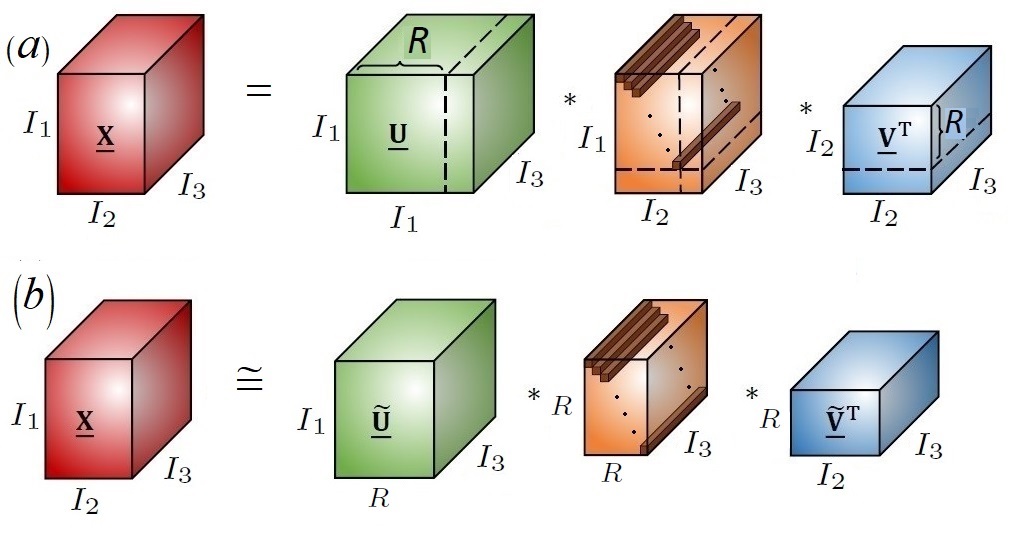}\\
\caption{\small{Illustration of ({\bf a}) Tensor SVD (t-SVD) and ({\bf b}) truncated t-SVD for a third-order tensor.}}\label{Pic7}
\end{center}
\end{figure}

Let $\underline{\bf X}\in\mathbb{R}^{I_1\times I_2\times I_3}$, then the t-SVD of the tensor $\underline{\bf X},$ admits the model 
\begin{eqnarray}\label{tmodel}
\underline{\bf X} \approx \underline{\bf U} * \underline{\bf S}* {\underline{\bf V}^T},
\end{eqnarray} 
where $\underline{\bf U}\in\mathbb{R}^{I_1\times R\times I_3}$, $\underline{\bf V}\in\mathbb{R}^{R\times I_2\times I_3}$ are orthogonal tensors and the tensor $\underline{\bf S}\in\mathbb{R}^{R\times R\times I_3}$ is f-diagonal (\cite{kilmer2011factorization,kilmer2013third}). This is also called the truncated t-SVD.

The number of nonzero tubes of the diagonal tensor $\underline{\bf S}$ is called the tubal rank. Unlike the Tucker decomposition or CPD, the truncated t-SVD gives the best approximation in the least-squares sense for any unitary invariant tensor norm \cite{kilmer2011factorization}. 

Computation of the t-SVD is performed in the Fourier domain and this is summarized in Algorithm \ref{ALG:t-SVD}. Here, similar to the t-product, it is sufficient to consider only the truncated SVD of the first $\lceil\frac{I_3+1}{2}\rceil$ frontal slices $\widehat{\X}(:,:,i),\,\,i=1,2,\ldots,\lceil\frac{I_3+1}{2}\rceil$, (see \cite{lu2019tensor}, for details). More precisely, we first compute the truncated SVD of the first $\lceil\frac{I_3+1}{2}\rceil$ frontal slices ${\widehat{\underline{\mathbf X}}(:,:,i)},\,i=1,2,\ldots,\lceil\frac{I_3+1}{2}\rceil$ as follows
\begin{eqnarray}\label{TRI}
\nonumber
[\widehat{\U}(:,:,i),\widehat{\tS}(:,:,i),\widehat{\V}(:,:,i)]&=& \\
\nonumber&&\hspace{-2cm}
{\rm truncated\_ svd}\left( {\widehat{\underline{\mathbf X}}(:,:,i),R} \right)\\&&\hspace{-2cm}
 i=1,2,\ldots,\lceil\frac{I_3+1}{2}\rceil,
\end{eqnarray}
and store $\widehat{\U}(:,:,i),\,\widehat{\V}(:,:,i)$ and $\widehat{\tS}(:,:,i),$ for $i=1,2,\ldots,\lceil\frac{I_3+1}{2}\rceil$. Then, they are used to recover the remaining factors $\widehat{\U}(:,:,i),\,\widehat{\tS}(:,:,i),\,\,\widehat{\V}(:,:,i),\,i=\lceil \frac{I_3+1}{2}\rceil+1,\ldots,I_3$ based on the following relations 
\begin{eqnarray}\label{TRI_2}
\nonumber
\widehat{\U}(:,:,i)&=&{\rm conj}(\widehat{\U}(:,:,I_3-i+2)),\\
\nonumber
 \widehat{\tS}(:,:,i)&=&\widehat{\tS}(:,:,I_3-i+2),\\
 \nonumber
\widehat{\V}(:,:,i)&=&{\rm conj}(\widehat{\V}(:,:,I_3-i+2)),\\
 &&i=\lceil \frac{I_3+1}{2}\rceil+1,\ldots,I_3.
\end{eqnarray}
Similar to the t-product, instead of the discrete Fourier transform matrices, one can define the t-SVD according to an arbitrary unitary transform matrix. It is shown in \cite{jiang2020framelet} that this can provide t-SVD with a lower tubal rank. 

\RestyleAlgo{ruled}
\LinesNumbered
\begin{algorithm}
\SetKwInOut{Input}{Input}
\SetKwInOut{Output}{Output}\Input{A data tensor $\underline{\mathbf X} \in {\mathbb{R}^{{I_1} \times {I_2} \times {I_3}}}$ and a target tubal rank $R$;} 
\Output{${\underline{\mathbf U}_R} \in {\mathbb{R}^{{I_1} \times R \times {I_3}}},\,\,{\underline{\mathbf S}_R} \in {\mathbb{R}^{R \times R \times {I_3}}},$ ${\underline{\mathbf V}_R} \in {\mathbb{R}^{{I_2} \times R \times {I_3}}}$;}
\caption{Truncated t-SVD}\label{ALG:t-SVD}
      {
$\widehat{\underline{\mathbf X}}= {\rm fft}\left( {\underline{\mathbf X},[],3} \right)$;\\
\For{$i=1,2,\ldots,I_3$}
{                        
$[{\mathbf U},{\mathbf S},{\mathbf V}] = {\rm truncated\_ svd}\left( {\widehat{\underline{\mathbf X}}(:,:,i),R} \right)$;\\
${\widehat{\underline{\mathbf U}}}\left( {:,:,i} \right) = {{\mathbf U}}$;\\
${\widehat{\underline{\mathbf S}}}\left( {:,:,i} \right) = {{\mathbf S}}$;\\ 
${\widehat{\underline{\mathbf V}}}\left( {:,:,i} \right) = {{\mathbf V}}$;\\
}
${\underline{\mathbf U}}_R = {\rm ifft}\left( {{\widehat{\underline{\mathbf U}}},[],3} \right),\,{\underline{\mathbf S}}_R = {\rm ifft}\left( {{\widehat{\underline{\mathbf S}}},[],3} \right),$\\
${\underline{\mathbf V}}_R = {\rm ifft}\left( {\widehat{\underline{\mathbf V}},[],3} \right).$  
       	}       	
\end{algorithm} 

The t-QR decomposition can be analogously defined based on the t-product. Here, for the t-QR decomposition of a tensor $\underline{\mathbf X}\in\mathbb{R}^{I_1\times I_2\times I_3},\,$ i.e., $\underline{\mathbf X} = \underline{\mathbf Q} * \underline{\mathbf R}$, we first compute the FFT of the tensor $\underline{\mathbf X}$ as 
\begin{equation}\label{fftten}
\underline{\widehat{\mathbf X}}={\rm fft}({\underline{\mathbf X}},[],3),
\end{equation}
and then the QR decomposition of all frontal slices of the tensor $\underline{\widehat{\mathbf X}}$ are computed as follows
\[
\underline{\widehat{\mathbf X}}(:,:,i)=\underline{\widehat{\mathbf Q}}(:,:,i)\,\,\underline{\widehat{\mathbf R}}(:,:,i),\,\,i=1,2,\ldots,I_3.
\]
Finally the IFFT operator is applied to the tensors $\underline{\widehat{\mathbf Q}}$ and $\underline{\widehat{\mathbf R}}$, to compute the tensors $\underline{{\mathbf Q}}$ and $\underline{{\mathbf R}}$. Similar to the matrix case, a tensor can be orthogonalized by applying the t-QR decomposition and taking the $\Q$ part. We use the notation ${\rm orth}(\X)$ to denote this operation. The tensor SVD and tensor QR decomposition with $I_3=1$ are reduced to the classical matrix SVD and QR decomposition. So, in this special case, ${\rm orth}({\bf X})$ gives the orthogonal part of the matrix QR decomposition. In \MATLAB, this is equivalent to $[{\bf Q},\sim]={\bf qr}({\bf X},0)$).

For large-scale tensors, Algorithm \ref{ALG:t-SVD} is expensive even if we use formulations \eqref{TRI} and \eqref{TRI_2}, because of the computation of a sequence of SVDs. To tackle this difficulty, the random projection technique can be used \cite{zhang2018randomized,qi2021t}. {To do so, in the first stage, we make reduction on the data tensor $\underline{\bf X}$ as $\underline{\bf Y} = \underline{\bf X}*\underline{\bf \Omega}$ where ${\bf\Omega} \in {\mathbb{R}^{{I_2} \times {(R+P)} \times {I_3}}}$ is a random tensor. Then, the t-QR decomposition of the tensor $\underline{\bf Y}$ is computed to get a low tubal rank approximation of the tensor $\X$ as follows
\begin{equation}\label{ltrank}
\underline{\bf X} \cong \underline{\bf Q}*\underline{\bf B},
\end{equation}
where $\underline{\bf B} = {\underline{\bf Q}^T}*\underline{\bf X},\,\,\,\underline{\bf Q} \in {\mathbb{R}^{{I_1} \times (R+P) \times {I_3}}},$ and $\underline{\bf B} \in {\mathbb{R}^{(R+P) \times {I_2} \times {I_3}}}$. Note $R+P\ll{I_2}$, where $R$ is an estimation of the tubal-rank and $P$ is the {\it oversampling} parameter for a better capturing the range of the tensor $\X$ \cite{zhang2018randomized}. Having computed the low tubal-rank approximation \eqref{ltrank}, and the t-SVD of $\B$ as $\underline{\bf B} = \widehat{\underline{\bf U}}*\underline{\bf S}*{\underline{\bf V}^T},$ the t-SVD of $\underline{\bf X}$ can be recovered as $\underline{\bf X} = \left( {\underline{\bf Q}*\widehat{\underline{\bf U}}} \right) * \underline{\bf S} * {\underline{\bf V}^T}.$ Note that the tensor $\B$ has smaller size than the original data tensor $\X$ and it is easier to handle. If we apply other decompositions such as t-QR or tensor CUR \cite{ahmadi2021cross,tarzanagh2018fast} on the compressed tensor $\B$, they are called randomized t-QR or randomized tensor CUR algorithms.}

A closely related concept is the {\it power iteration} technique which is used in the scenarios that the frontal slices of the underlying data tensor $\underline{\bf X}$ do not have fast decaying singular values. In this situation, instead of the data tensor $\underline{\bf X}$, the new tensor $\underline{\bf Y} = {\left( {\underline{\bf X} * {\underline{\bf X}^T}} \right)^q} * \underline{\bf X}$ is multiplied with a random tensor. By straightforward computations and substituting the t-SVD of $\underline{\bf X}$  \eqref{tmodel} in $\underline{\bf Y}$, we have $\underline{\bf Y} = \underline{\bf U} * {\underline{\bf S}^{2q + 1}} * {\underline{\bf V}^T}$, where ${\underline{\bf S}^{2q + 1}}=\underbrace{{\tS}*{\tS}*\cdots *{\tS}}_{2q+1\,\,{\rm times}}$. It is seen that the left and right tensor parts of the t-SVD of $\underline{\bf X}$ and $\underline{\bf Y}$ are the same while the middle tensor of $\underline{\bf Y}$, e.g. $\underline{\bf S}^{2q + 1}$ has frontal slices with faster decaying singular values than the tensor ${\tS}$. For the stability issues, we should avoid the direction computation $\underline{\bf Y} = {\left( {\underline{\bf X} * {\underline{\bf X}^T}} \right)^q} * \underline{\bf X}$ and this should be performed sequentially using the subspace iteration method \cite{zhang2018randomized}, (we will use this approach in Section \ref{ProApp}). Larger $P$ and $q$ lead to better accuracy while they need higher computational costs. In practice, the power iteration $q=1,2,3$ and oversampling $P=10$ are often sufficient to provide satisfying results \cite{halko2011finding} .

The sampling approach can also be used for low tubal rank approximation besides the random projection. Indeed, a randomized slice sampling algorithm was proposed in \cite{tarzanagh2018fast} in which horizontal and lateral slices are selected and a low tubal rank approximation is computed based on them, see Figure \ref{Pic2} for a graphical illustration on this approach. The work \cite{qi2021t}, generalized the pass-efficient randomized algorithms proposed in \cite{tropp2017practical} to tensors. These algorithms are very efficient when the number of passes over the data tensor is our main concern. For example in situations that the underlying data tensor is stored out of core and the communication cost may exceed our main computations. 
All of these algorithms need an estimation of the tubal rank and for a given error bound, they can not find an optimal tubal rank automatically.

\begin{figure}
\begin{center}
\includegraphics[width=8.75 cm,height=3.75 cm]{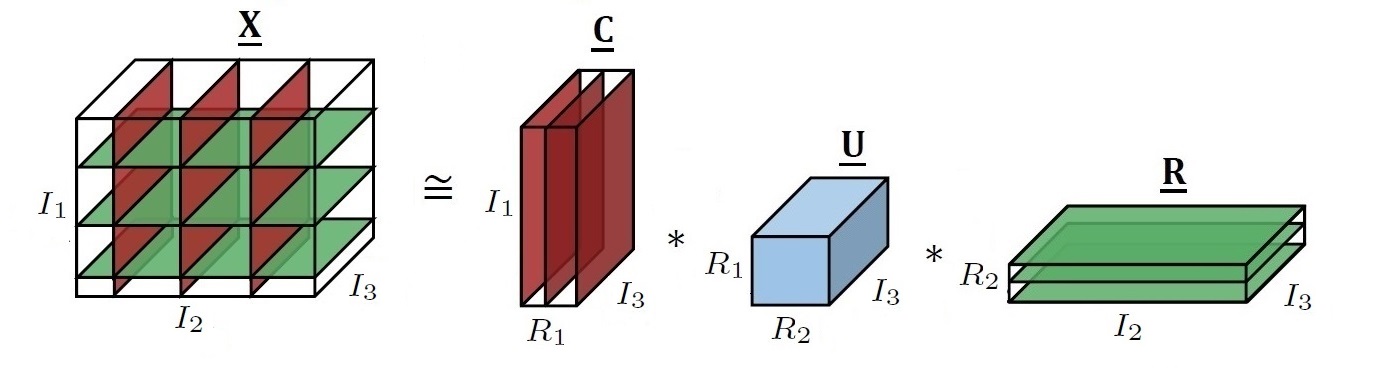}\\
\caption{\small{Randomized t-SVD based on slice sampling \cite{tarzanagh2018fast,ahmadi2021cross}.}}\label{Pic2}
\end{center}
\end{figure}

We now derive the adjoint of the operator $\mathcal{L}\left( \underline{\bf X} \right) =\underline{\bf Q} * \underline{\bf X}$ which we will need it in our analysis.

\begin{lem}\label{Lem1}
For the linear operator defined as follow
\[
\begin{array}{l}
\mathcal{L}:{\mathcal{V}} \to {\mathcal{V}}\\
\\
\mathcal{L}\left( \underline{\bf X} \right) = \underline{\bf Q} * \underline{\bf X},
\end{array}
\]
where $\mathcal{V}$ is the space of third-order tensors,
the adjoint of $\mathcal{L}$ is ${\mathcal{L}^{adj}}\left( \underline{\bf X} \right)=\underline{\bf Q}^T*\underline{\bf X}.$ 
\end{lem}
\begin{proof}
From the following straightforward computations, we have
\begin{eqnarray}\label{Pr1}
\nonumber
\left\langle {\underline{\bf Q} * \underline{\bf X},\underline{\bf Y}} \right\rangle  &=& \left\langle {{\rm circ}\left( \underline{\bf Q} \right){\rm unfold}\left( \underline{\bf X} \right),{\rm unfold}\left( \underline{\bf Y} \right)} \right\rangle \\
\nonumber
 &=& \left\langle {{\rm unfold}\left( \underline{\bf X} \right),{\rm circ}{{\left( \underline{\bf Q} \right)}^T}{\rm unfold}\left( \underline{\bf Y} \right)} \right\rangle \\
 &=& \left\langle {\underline{\bf X},{\underline{\bf Q}^T} * \underline{\bf Y}} \right\rangle.
\end{eqnarray}
In deriving \eqref{Pr1}, note ${\rm circ}\left( {{\underline{\bf Q}^T}} \right) = {\left( {{\rm circ}\left( \underline{\bf Q} \right)} \right)^T}.$ This completes the proof.
\end{proof}

Using Lemma \eqref{Lem1}, it can be easily proved that multiplying a tensor with an orthogonal tensor does not change its Frobenius norm because
\begin{eqnarray*}
\left\| {\underline{\bf Q} * \underline{\bf B}} \right\|_F^2 = \left\langle {\underline{\bf Q} * \underline{\bf B},\underline{\bf Q} * \underline{\bf B}} \right\rangle  = \left\langle {{\underline{\bf Q}^T} * \underline{\bf Q} * \underline{\bf B},\underline{\bf B}} \right\rangle  = \left\| \underline{\bf B} \right\|_F^2.
\end{eqnarray*}

\section{Proposed approach}\label{ProApp}
As discussed earlier, the randomized algorithms proposed in (\cite{zhang2018randomized,tarzanagh2018fast,qi2021t,ding2022randomized}) require an estimation of the tubal rank which may be a difficult task. To solve this limitation, we propose a new randomized fixed-precision or adaptive algorithm which for a given approximation error bound, it can find an optimal tubal rank of the underlying data tensor and corresponding low tubal rank approximation, automatically. {To be more precise, let $\X\in\mathbb{R}^{I_1\times I_2\times I_3}$ be a given data tensor and $\epsilon$ be a given error bound. We aim at finding a tubal rank $R$ for which we have
\begin{equation}\label{Err_1}
\|\X-\Q*\B\|\leq \epsilon,
\end{equation}
where $\Q\in\mathbb{R}^{I_1\times R\times I_3}$ and $\B=\Q^T*\X\in\mathbb{R}^{R\times I_2\times I_3}$. Then the t-SVD of the tensor $\X$  can be recovered from computing the truncated t-SVD of the smaller tensor $\B$ as discussed in Section \ref{Sec:TSVD}.}

{
Our algorithm is a generalization of the randomized fixed-precision algorithm  developed in \cite{yu2018efficient} which is a modification and improved version of the fixed-precision algorithm proposed in \cite{martinsson2016randomized}. Similar to the idea presented in \cite{yu2018efficient}, we propose to gradually increase the size of the second and first modes of $\Q$ and $\B$, respectively, until the condition \eqref{Err_1} is satisfied. In terms of the terminologies introduced in Section \ref{Pre}, this means that the new block tensors are concatenated along the second and the first modes of the current blocks $\Q$ and $\B$, respectively. For a better understanding of the proposed algorithm, let us first start with the randomized blockQB algorithm developed in \cite{martinsson2016randomized} for low rank matrix approximation and summarized in Algorithm \ref{ALgRRM}. For a given matrix ${\bf X}\in\mathbb{R}^{I\times J}$ and a target rank $R$, the random projection method, provides a low rank approximation ${\bf X}\cong{\bf Q}{\bf B}$ as follows
\begin{eqnarray}\label{R_P}
\nonumber
{\bf \Omega }&=&{\rm randn}(J,R),\\
\nonumber
{\bf Q} &=& {\rm orth}({\bf X}{\bf\Omega}),\\
{\bf B} &=& {\bf Q}^T{\bf X}.
\end{eqnarray}
The block form of procedure \eqref{R_P} improves computational efficiency, simplifies implementation on parallel machines and makes it possible to use the BLAS-3 operations. Besides, it facilitates the adaptive rank determination (\cite{martinsson2016randomized,halko2011finding}).
Let $R=bs$ where $b$ is the block size parameter and $s$ is the number of such blocks to build matrices. In practice, $b$ is selected much smaller than $R$ while this depends strongly on the underlying hardware \cite{yu2018efficient}.}
{
Then, consider the block versions of the random matrix ${\bf \Omega}$ and the matrix ${\bf B}$ as ${\bf \Omega}=[{\bf \Omega}^{(1)},{\bf \Omega}^{(2)},\ldots,{\bf \Omega}^{(s)}]$ and ${\bf B}=[{\bf B}^{(1)\,T},{\bf B}^{(2)\,T},\ldots,{\bf B}^{(s)\,T}]^T$, where ${\bf \Omega}^{(i)}\in\mathbb{R}^{J\times b}$ and ${\bf B}^{(i)}\in\mathbb{R}^{b\times J}$ for $i=1,2,\ldots,s.$ It was proved in \cite{martinsson2016randomized} that the following iterative procedures 
\begin{eqnarray}
\nonumber
{\bf Q}^{(i)}&=&{\rm orth}{\left({\bf X}^{(i-1)}{\bf \Omega}^{(i)}\right)},\\
\nonumber
{\bf B}^{(i)}&=&{\bf Q}^{(i)\,T}{{\bf X}^{(i-1)}},\\
{{\bf X}^{(i)}}&=&{{\bf X}^{(i-1)}}-{\bf Q}^{(i)}{\bf B}^{(i)},
\end{eqnarray} 
are equivalent to \eqref{R_P}  where ${\bf X}^{(0)}={\bf X}$ and $i=1,2,\ldots,s,$ but in the block format. This blocking strategy helps to develop an adaptive algorithm for low rank approximation because, at each iteration, we can add new blocks to the matrices ${\bf Q}$ and ${\bf B}$ and check whether the stopping criterion $\|{\bf X}-{\bf Q}{\bf B}\|\leq \epsilon$ is satisfied or not. So, we can adaptively estimate the rank.}

Algorithm \ref{ALgRRM}, is the implementation of the above-mentioned procedure which is also equipped with the power iteration (Lines 4-7) for the situations that the singular values do not decay very fast.  

The authors in \cite{yu2018efficient} proposed a more efficient version of Algorithm \ref{ALgRRM} in the sense of computing the error norm at each iteration. Our proposed Algorithm \ref{ALgRR} is a generalization of this algorithm to the tensor case. 
 It requires only an approximation error bound $\epsilon$, a block size $b$ and a power iteration parameter $q$ as inputs. Also the operator ``orth'' in Lines 5, 7, 8 and 10, gives the orthogonal part of the t-QR decomposition for a given data tensor. {Please note that in the sequel, the notation $\B^{(i)},$ denotes a block tensor constructed by composing $i$ tensors $\B^{(k)},\,k=1,2,\ldots,i$, i.e., $\underline{\bf B}_i=\begin{bmatrix}
\B^{(1)}\\
\B^{(2)}\\
\vdots\\
\B^{(i)}
\end{bmatrix}.$ }

\RestyleAlgo{ruled}
\begin{algorithm}
\LinesNumbered
\SetKwInOut{Input}{Input}
\SetKwInOut{Output}{Output}  \Input{A matrix ${\bf X}\in\mathbb{R}^{I_1\times I_2}$; an error bound $\epsilon$; a block size $b$.} \Output{${\bf Q}=[{\bf Q}^{(1)},\ldots,{\bf Q}^{(i)}],\,{\bf B}=[{\bf B}^{(1)\,T},\ldots,{\bf B}^{(i)\,T}]^T$ such that ${\left\| {{\bf X} - {\bf Q}{\bf B}} \right\|_F} < \varepsilon$ and corresponding optimal matrix rank $R$}
\caption{The blocked randQB algorithm \cite{martinsson2016randomized}}\label{ALgRRM}
\For{$i=1,2,3,\ldots$}{
${{\bf\Omega}^{(i)}} = {\rm randn}\left( {I_2,b} \right)$;\\
${{\bf Q}^{(i)}} = {\rm orth}\left( {{\bf X}{{\bf \Omega}^{(i)}}} \right)$;\\
\For{$j=1,2,\ldots,q$}{
${\bf Q}^{(i)}={\rm orth}({\bf X}^T{\bf Q}^{(i)})$;\\
${\bf Q}^{(i)}={\rm orth}({\bf X}{\bf Q}^{(i)})$;\\
}
${{\bf Q}^{(i)}} = {\rm orth}\left( {{\bf Q}^{(i)}}-\sum_{j=1}^{i-1}{\bf Q}^{(j)}{\bf Q}^{(j)\,T}{\bf Q}^{(i)}\right)$;\\
${{\bf B}^{(i)}} = {\bf Q}^{(i)T}{\bf X}$;\\
${\bf X}={\bf X}-{\bf Q}^{(i)}{\bf B}^{(i)}$;\\
$
{\bf Q}=[{\bf Q}^{(1)},\ldots,{\bf Q}^{(i)}]$;\\
${\bf B}=[{\bf B}^{(1)\,T},\ldots,{\bf B}^{(i)\,T}]^T$;\\
if {$\|{\bf X}\|_F\leq \epsilon$}
{then stop}
}
\end{algorithm} 

\RestyleAlgo{ruled}
\begin{algorithm}
\LinesNumbered
\SetKwInOut{Input}{Input}
\SetKwInOut{Output}{Output}  \Input{A tensor $\underline{\bf X}\in\mathbb{R}^{I_1\times I_2\times I_3}$; an error bound $\epsilon$; a block size $b$ and a power iteration $q$.}  \Output{$\underline{\bf Q}=[\Q^{(1)},\Q^{(2)},\ldots,\Q^{(i)}],\,\underline{\bf B}=\begin{bmatrix}
\B^{(1)}\\
\B^{(2)}\\
\vdots\\
\B^{(i)}
\end{bmatrix}
$ such that ${\left\| {\underline{\bf X} - \underline{\bf Q}*\underline{\bf B}} \right\|_F} < \varepsilon$ and corresponding optimal tubal rank $R$}
$\underline{\bf Q}=[],\,\,\underline{\bf B}=[]$;\\
\caption{Proposed randomized fixed-precision algorithm}\label{ALgRR}
$E = \left\| \underline{\bf X} \right\|_F^2$;\\
\For{$i=1,2,\ldots$}{
${\underline{\bf\Omega}^{(i)}} = {\rm randn}\left( {I_2,b,I_3} \right)$;\\
${\underline{\bf Q}^{(i)}} = {\rm orth}\left( {\underline{\bf X}*{\underline{\bf \Omega}^{(i)}} - \underline{\bf Q}*\left( {\underline{\bf B}*{\underline{\bf\Omega}^{(i)}}} \right)} \right)$;\\
\For{$j=1,2,\ldots, q$}
{${\underline{\bf Q}^{(i)}} = {\rm orth}\left( {{\underline{\bf X}}^T*{\underline{\bf Q}^{(i)}}} \right)$;\\
${\underline{\bf Q}^{(i)}} = {\rm orth}\left( {\underline{\bf X}*{\underline{\bf Q}^{(i)}}} \right)$
}
${\underline{\bf Q}^{(i)}} = {\rm orth}\left( {{\underline{\bf Q}^{(i)}} - \underline{\bf Q}*\left( {{\underline{\bf Q}^T}*{\underline{\bf Q}^{(i)}}} \right)} \right)$;\\
${\underline{\bf B}^{(i)}} = \underline{\bf Q}^{(i)T}*\underline{\bf X}$;\\
$
\underline{\bf Q}=\underline{\bf Q} \boxplus_2{\underline{\bf Q}^{(i)}}
$\\
$\underline{\bf B}=\underline{\bf B} \boxplus_1 {\underline{\bf B}^{(i)}}
$\\
$E = E - \left\| {{\underline{\bf B}^{(i)}}} \right\|_F^2$;\\
\If{$E < {\varepsilon ^2}$}{Break}
       	}
\end{algorithm}

\RestyleAlgo{ruled}
\begin{algorithm}
\caption{Precise tubal rank selection (Lines 15-17 of Algorithm \ref{ALgRR})}\label{Rank_sel}
\If{$\left(E - \left\| {{\underline{\bf B}^{(i)}}} \right\|_F^2<\epsilon^2\right)$}{
\For{$j=1,2,\ldots,b$}
{$E = E - \left\| {{\underline{\bf B}^{(i)}}}(j,:,:) \right\|_F^2$;\\
\If{$E<\epsilon^2$}
{Remove ${{\underline{\bf B}^{(i)}}}(j+1:b,:,:)$ from 
${{\underline{\bf B}^{(i)}}}$ and remove 
${{\underline{\bf Q}^{(i)}}}(:,j+1:b,:)$ from ${{\underline{\bf Q}^{(i)}}}$
;\\
{\bf Break}
}
}
\eIf{$j<b$}
 {Set the tubal rank as $(i-1)b+j$}
 {Set the tubal rank as $ib$}
 }
\end{algorithm} 

Inspired by \cite{yu2018efficient}, in the next theorem, we present an elegant way for computing the residual error term.
\begin{thm}\label{ThmRU}
Let $\underline{\bf X}\in\mathbb{R}^{I_1\times I_2\times I_3}$ be a given tensor, $\underline{\bf Q}\in\mathbb{R}^{I_1\times R\times I_3}$ be an orthogonal tensor $(R<I_2),$ and $\underline{\bf B} = {\underline{\bf Q}^T}*\underline{\bf X}\in\mathbb{R}^{{R \times {I_2} \times {I_3}}}$ then
\[
\left\| {\underline{\bf X} - \underline{\bf Q}* \underline{\bf B}} \right\|_F^2 = \left\| \underline{\bf X} \right\|_F^2 - \left\| \underline{\bf B} \right\|_F^2.
\]
\end{thm}
\begin{proof}
Considering the following straightforward computations
\begin{eqnarray*}
\left\| {\underline{\bf X} - \underline{\bf Q}*\underline{\bf B}} \right\|_F^2 &=& \left\langle {\underline{\bf X} - \underline{\bf Q}*\underline{\bf B},\underline{\bf X} - \underline{\bf Q}*\underline{\bf B}} \right\rangle  \\
&=& \left\| \underline{\bf X} \right\|_F^2 -2 \left\langle {\underline{\bf X},\underline{\bf Q}*\underline{\bf B}} \right\rangle +\left\| \underline{\bf B} \right\|_F^2 \\
&=& \left\| \underline{\bf X} \right\|_F^2 -2 \left\langle {\underline{\bf Q}^T*\underline{\bf X},\underline{\bf B}} \right\rangle +\left\| \underline{\bf B} \right\|_F^2 \\
&=&\left\| \underline{\bf X} \right\|_F^2 - \left\| \underline{\bf B} \right\|_F^2.
\end{eqnarray*}
the theorem is proved. 
\end{proof}

From Theorem \ref{ThmRU}, instead of computing the residual $\underline{\bf X}-\underline{\bf Q}*\underline{\bf B}$ and  its Frobenius norm, we only need to compute the Frobenius norm of the tensor $\underline{\bf B}$ at each iteration which has a smaller size ($R \times  {I_2}\times I_3$). The Frobenius norm of the original data tensor $\underline{\bf X}$ is computed once and the residual is computed by subtracting it from $\|\underline{\bf B}\|^2_F$. This computation can be further reduced as stated in Theorem \ref{THM1}.

\begin{thm}\label{THM1} After executing the $i$-th iteration of Algorithm \ref{ALgRR}, the error term $E_i$ is 
\[
{E_{i}} = {E_{i-1}} - \left\| {{\underline{\bf B}^{(i)}}} \right\|_F^2,\ i=1,2,\dots
\]
where 
${E_{i}} = \left\| {\underline{\bf X} - {\underline{\bf Q}_{i}}*{\underline{\bf B}_{i}}} \right\|_F^2$, 
and $\underline{\bf Q}_i=[\underline{\bf Q}^{(1)},\underline{\bf Q}^{(2)},\ldots,\underline{\bf Q}^{(i)}],\,\,\,\,\,\underline{\bf B}_i=\underline{\bf B}_i=\begin{bmatrix}
\B^{(1)}\\
\B^{(2)}\\
\vdots\\
\B^{(i)}
\end{bmatrix}$.
\end{thm}

\begin{proof}
By the following straightforward computations and using Theorem \ref{ThmRU}, we have
\begin{eqnarray*}
{ E}_{i} &=& \left\| \underline{\bf X} \right\|_F^2 - \left\| {{\underline{\bf B}_{i}}} \right\|_F^2 
= \left\| \underline{\bf X} \right\|_F^2 - \sum\limits_{j = 1}^{i-1} {\left\| {{\underline{\bf B}^{(j)}}} \right\|_F^2} 
- \left\| {{\underline{\bf B}^{(i)}}} \right\|_F^2 \\
&=& \left\| \underline{\bf X} \right\|_F^2 - \left\| {{\underline{\bf B}_{i-1}}} \right\|_F^2
- \left\| {{\underline{\bf B}^{(i)}}} \right\|_F^2
= {E_{i-1}} - \left\| {{\underline{\bf B}^{(i)}}} \right\|_F^2,
\end{eqnarray*}
where we have used the property ${\underline{\bf B}_{i}} = {\underline{\bf Q}_i^{T}}*\underline{\bf X}$. This completes the proof.
\end{proof}

{
Assume that at iteration $i$ of Algorithm \ref{ALgRR}, the stopping criterion is satisfied. This means we have already constructed the block tensors $\underline{\bf B}_i=\begin{bmatrix}
\B^{(1)}\\
\B^{(2)}\\
\vdots\\
\B^{(i)}
\end{bmatrix},$ and $\underline{\bf Q}_i=[\Q^{(1)},\Q^{(2)},\ldots,\Q^{(i)}].$ Then to select the tubal rank precisely, we consider the last block tensors $\B^{(i)}$ and $\Q^{(i)}$ and gradually remove horizontal and lateral slices from them respectively. This procedure is described in Algorithm \ref{Rank_sel}. In view of Theorems \ref{ThmRU} and \ref{THM1}, we do not need to consider the block tensor $\Q^{(i)}$ and only we need to remove horizontal slices from the tensor $\B^{(i)}$.
}

\section{Computational complexity}\label{sec:com}
As discussed in \cite{zhang2018randomized}, the computational complexity of Algorithm \ref{ALG:t-SVD} for a tensor of Size $I\times J\times K$ is 
$\mathcal{O}(IJK{\rm log}(K))+\mathcal{O}(IJK{\rm min}(I,J))$ where the first term is for transforming 
the tensor to the Fourier domain while the second is for tensor decomposition. In contrast, the 
complexity of randomized Algorithm \ref{ALgRR} is $\mathcal{O}(IJKb)$ which for $b\ll {\rm min}\{I,J\}$ is significantly lower.

\section{Simulations}\label{Sec:Sim}
In this section, we test the proposed randomized algorithm using synthetic and real-world data. All numerical simulations were performed on a laptop computer with 2.60 GHz Intel(R) Core(TM) i7-5600U processor and 8GB memory. In all our simulations, we set the power iteration $q=1$. The compression ratio is defined as $\frac{\rm Number\,of\,parameters\,of\,the\,original\,tenors}{\rm Number\,of\,parameters\,in\,the\,compressed\,model}$. We have used some \textsc{Matlab} function from the toolbox \url{https://github.com/canyilu/Tensor-tensor-product-toolbox}. The codes are available at \url{https://github.com/SalmanAhmadi-Asl/Fixed-Precision-t-SVD}.
\begin{exa}\label{Ex1}
In this example, we apply our algorithm to a synthetic data tensor of low tubal-rank. Let $\underline{\bf U},\,\underline{\bf V}\in\mathbb{R}^{n\times n\times n},$ be random orthogonal tensors with respect to the t-product. To generate such tensors, we first generate random standard Gaussian tensors (zero mean and unit variance), then compute their t-QR decomposition and take the Q part in our simulations. Also for the middle tensor $\underline{\bf S}\in\mathbb{R}^{n\times n\times n}$, we consider the following three cases
\begin{itemize}
\item Case I. The diagonal tensor $\underline{\bf S}$ with only $R$ nonzero tubes. The components of all tubes were also standard Gaussian.

\item  Case II. (Fast polynomial decay) \cite{tropp2017practical} Consider the frontal slices of the tensor $\underline{\bf S}$ defined as follows
\begin{eqnarray*}
\underline{\bf S}(:,:,i)={\rm diag}\left( \underbrace {1, \ldots ,1}_R,{2^{ - 2}}{,3^{ - 2}}{,4^{ - 2}}, \ldots,\right.\\
\left.{\left( {n - R + 1} \right)^{ - 2}}
 \right),\,\,i=1,2,\ldots,n.
\end{eqnarray*}


\item Case III. (Fast exponentially decay) \cite{tropp2017practical} Consider the frontal slices of the tensor $\underline{\bf S}$ defined as follows
\begin{eqnarray*}
\underline{\bf S}(:,:,i)={\rm diag}\left( \underbrace {1, \ldots ,1}_R,{10^{ - 1}}{,10^{ - 2}}{,10^{ - 3}}, \ldots, \right.\\
\left.{10^{ - \left( {n - R} \right)}}
 \right),\,\,i=1,2,\ldots,n.
\end{eqnarray*}


\end{itemize}
We generate a new tensor corrupted with a noise term as follows
\[
\underline{\bf X} = \underline{\bf U}*\underline{\bf S}*\underline{\bf V}^T+\delta\,\frac{\underline{\bf Y}}{{\left\|{\underline{\bf Y}}\right\|}},
\]
where $\underline{\bf Y}\in\mathbb{R}^{n\times n\times n},\,n=100,150,\ldots,500$ is a standard Gaussian data tensor with zero mean and unit variance and $\delta=0.01$. We set parameter $R=10$ for tensor Cases I-III and apply Algorithm \ref{ALgRR} (with $\epsilon=0.01$ and block size $b=100$) to them. The interesting point is that Algorithm \ref{ALgRR} estimated the true numerical tubal rank 10 for all tensors of different dimensions. This convinced us that the proposed algorithm is a reliable algorithm for tubal rank estimation. 
Then to examine the speed-up of our algorithm, we used the truncated t-SVD with the estimated tubal rank 10. The running times of the proposed algorithm and the truncated t-SVD algorithm for different dimensions for Cases I-III are reported in Figure \ref{figEX1}. The linear scaling of Algorithm \ref{ALgRR} compared to the truncated t-SVD is visible. In all our experiments, the truncated t-SVD provided better accuracy but still quite close to the accuracy achieved by of the proposed algorithm.


\begin{figure}
\begin{center}
\includegraphics[width=8 cm,height=6 cm]{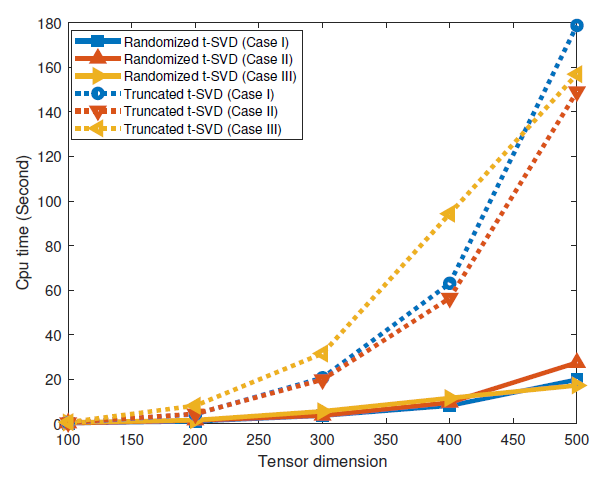}
\caption{\small{Example \ref{Ex1}. Running times comparison against tensor dimensions for Cases I-III with tubal rank $R=10$.}}\label{figEX1}
\end{center}
\end{figure} 

 
\end{exa}

%
%


\begin{exa}\label{Exa_2}
{
In this experiment, we consider two new synthetic third-order tensors to check again the performance of the proposed algorithm. Consider  third-order tensors of size $500\times 500\times 500$ whose components are generated as follows}
\begin{itemize}
    \item {Case I:\quad ${\X}(i,j,k)=\frac{1}{i+j+k}$,}\\
    
    \item {Case II:\quad ${\X}(i,j,k)=\frac{1}{\left({i^{0.5}+j^{0.5}+k^{0.5}}\right)^{1/{0.5}}}$,}\\
\end{itemize}
{where $1\leq i,j,k\leq 500$. The (numerical) tubal rank of the tensor Case I is 21 and it is $14$, for the tensor Case II. We have used the \MATLAB code {\sffamily tubalrank.m} included in {\it Tensor-Tensor Product Toolbox 2.0} to estimate the tubal rank of tensors. We set $\epsilon=0.001$ with block size $b=20$ and apply Algorithm \ref{ALgRR} to these synthetic tensors. Here, Algorithm \ref{ALgRR} gave the tubal rank 14 for the tensor Case I and the tubal rank 6 for the tensor Case II. Then we applied the truncated t-SVD to the data tensors (Case I and Case II) with the tubal ranks 14 and 6, respectively. Also, we considered the tolerances $\epsilon=0.01,\,0.10$ for which Algorithm \ref{ALgRR} estimated the tubal ranks as
\begin{itemize}
    \item For $\epsilon=0.01$, Case I (tubal rank 8) $\&$  Case II (tubal rank 4)
    \item For $\epsilon=0.1$, Case I (tubal rank 3) $\&$ Case II (tubal rank 2)
\end{itemize}
Then, we applied the truncated t-SVD with the mentioned tubal ranks to the underlying data tensors. The running times of the proposed algorithm and the truncated t-SVD are reported in Table \ref{tab:resnet_inference}. It is seen that the Algorithm \ref{ALgRR} outperforms the truncated t-SVD in terms of running time.}

\begin{table*}[!ht]
    \centering
    \caption{Example \ref{Exa_2}, The comparison of running times and relative errors of the proposed algorithm and the truncated t-SVD.}
    \label{tab:resnet_inference}
    \begin{tabular}{l c c}
    \hline
        \multirow{2}{*}{} & \multicolumn{2}{c}{Data Tensor Case I}\\\hline
         $\epsilon$ &  Proposed & Truncated t-SVD  \\\hline
         {$0.001$} & {\bf 9.12} Sec (Error=9.2512e-04) & 43.41 Sec (Error=2.4598e-04)\\
        {$0.01$} & {\bf 6.22} Sec (Error=0.0041) & 33.41 Sec (Error=0.0014) \\
        {$0.10$} & {\bf 4.57} Sec (Error=0.0572) & {22.53 Sec (Error=0.0252) }
        \\\hline
         \multirow{2}{*}{} & \multicolumn{2}{c}{Data Tensor Case II}\\\hline
        $\epsilon$  &  Proposed & Truncated t-SVD  \\\hline
        {$0.001$} & {\bf 11.20} Sec (Error7.9713e-04) & 44.32 Sec (Error=2.4598e-04)\\
        {$0.01$} & {\bf 7.23} Sec (Error=0.0076) & 35.13 Sec (Error=0.0027) \\
        {$0.10$} & {\bf 5.32} Sec (Error=0.0720) & {{21.13} Sec (Error=0.0313) }\\\hline
    \end{tabular}
\end{table*}


\end{exa}

\begin{exa}\label{Ex2}
In this example, we apply our proposed algorithm to compress the Yale B dataset. This dataset includes images of size $192\times 168$ for $38$ persons under $30$ different illumination conditions and as a results a tensor of size $192\times 168 \times 30\times 38$ is produced. We reshape this tensor to a third-order tensor of size $192\times 1140 \times 168$ and apply our proposed algorithm with approximation error bounds $\epsilon=0.05,\,0.04,\,0.03$. The approximate tubal ranks found by the our algorithm were $32,\,46,\,57$ with corresponding compression ratios $5.0147,\,3.4530,\,2.7646$, respectively. Then, for these specific tubal ranks, we also applied the truncated t-SVD algorithm. The running times and relative errors of the truncated t-SVD and the proposed algorithm are reported in Table \ref{tab:resnet}. The reconstructed images for the tubal rank $R=32$ and $\epsilon=0.05$ are depicted in Figure \ref{figEX3}. Here again, the superiority of the proposed algorithm over the truncated t-SVD in terms of running time (seconds) is visible.

\begin{table*}[!ht]
    \centering
    \caption{Example \ref{Ex2}, The comparison of running times and relative errors of the proposed algorithm and the truncated t-SVD.}
    \label{tab:resnet}
    \begin{tabular}{l c c}
        \multirow{2}{*}{} & \multicolumn{2}{c}{}\\\hline
         $\epsilon$ &  Proposed & Truncated t-SVD  \\\hline
                 {$0.03$} & {\bf 11.205} Sec (Error=0.0285) & 52.405 Sec (Error=0.0273) \\
      {$0.04$} & {\bf 9.0678} Sec (Error=0.0395) & 39.33 Sec (Error=0.0385) \\
         {$0.05$} & {\bf 7.7945} Sec (Error=0.0499 ) & 26.16 Sec (Error=0.0495)
        \\\hline
    \end{tabular}
\end{table*}


\begin{figure}
\begin{center}
\includegraphics[width=9 cm,height=5 cm]{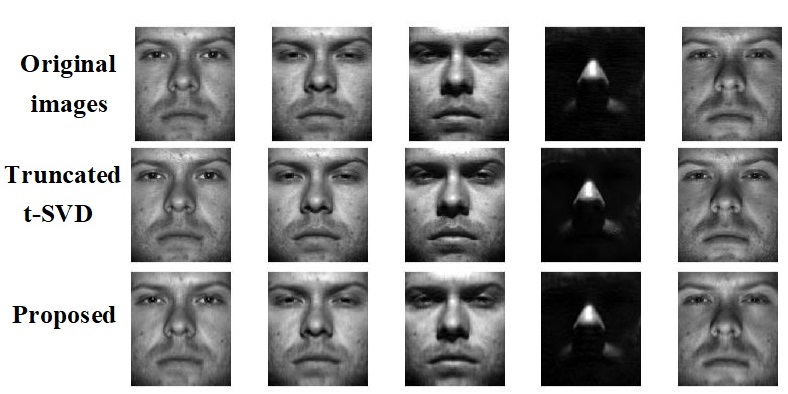}
\caption{\small{Example \ref{Ex2}. Reconstructed images with the tubal rank $R=35$ and $\epsilon=0.05$ using the truncated t-SVD and the proposed algorithm.}}\label{figEX3}
\end{center}
\end{figure}

\end{exa}


\begin{exa}\label{Ex4}
{
In this example, we apply Algorithm \ref{ALgRR} to the COIL-100 dataset \cite{nene1996columbia} which is the extension of the COIL-20 dataset. This data tensor consists of 7200 color images (100 objects under 72 rotations per object, see Figure \ref{Coil_sample} for some samples of this data tensor). The size of each image is $128\times 128 \times 3$ and the whole data is represented as a fifth-order tensor of size $128\times 128\times 3\times 100\times 72$. We only consider 30 objects under 30 different rotations, so we have a fifth order tensor of size $128\times 128\times 3 \times 30\times 30$. Since the proposed algorithm is applicable for third-order tensors, we first reshaped the tensor to a third-order tensor of size $384\times 384\times 300$, then applied the algorithm to compress it. Finally, the approximate tensors were reshaped back to the original size of the dataset. Here, we considered the error bound $\epsilon=9e-2$ and block size $b=50$ in our computations. Algorithm \ref{ALgRR} gave the tubal rank 45 and for this tubal rank, we applied the truncated t-SVD to compress the COIL100 dataset. We achieved again a significant acceleration ($\times 7.5$ speed-up) compared with the truncated t-SVD. The reconstructed images obtained by our proposed algorithm and the truncated t-SVD are visualized in Figure \ref{COIL_Res}. We see from Figure \ref{COIL_Res},  that the reconstructed images by the proposed algorithm are close to the ones computed by the truncated t-SVD but with much less running time. Also, we tried to use other tolerances to further examine the efficiency of  Algorithm \ref{ALgRR}. We found that for very small tolerances, the tubal rank may be high for which the randomized algorithms are not applicable. Due to this issue we considered the moderate tolerances $\epsilon=0.10,\,0.05,\,0.07$ for which Algorithm \ref{ALgRR} estimated the relatively small tubal ranks $39,\,96$ and $64$, respectively. The truncated t-SVD was then applied to the COIL-100 dataset with the mentioned tubal ranks. The running times and corresponding relative errors are reported in Table \ref{tab:resnet_inferenc}. The results clearly show that the proposed algorithm is more efficient and faster than the truncated t-SVD for low tubal rank approximation of tensors.  
}

\begin{figure}
\begin{center}
\includegraphics[width=1\linewidth
]{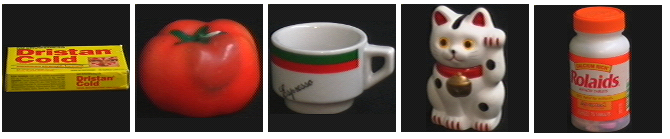}
\caption{\small{Some random samples of the COIL-100 dataset.}}\label{Coil_sample}
\end{center}
\end{figure}

\begin{figure}
\begin{center}
\includegraphics[width=.9\linewidth
]{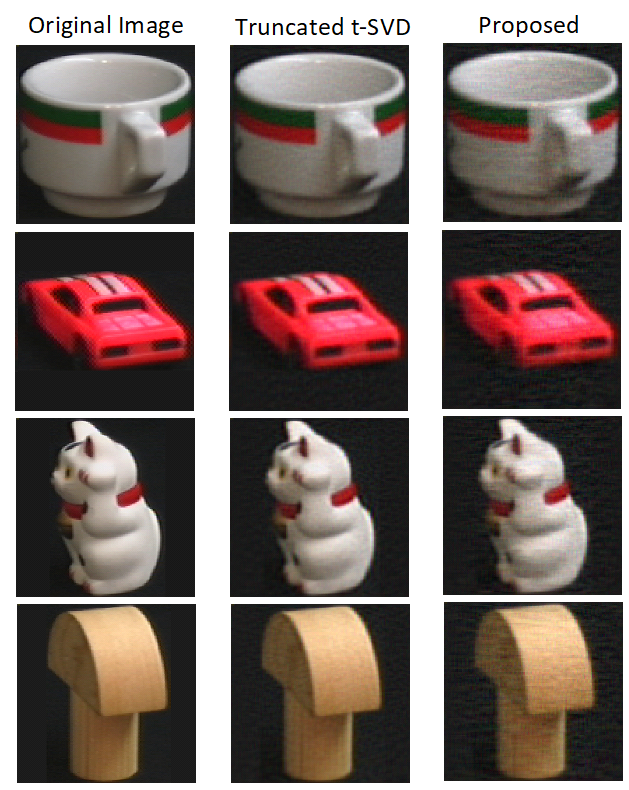}
\caption{\small{Example \ref{Ex4}, Reconstructed images with the tubal rank $R=45$ and $\epsilon=9e-2$ using the truncated t-SVD and the proposed algorithm.}}\label{COIL_Res}
\end{center}
\end{figure}
\end{exa}

\begin{table*}[!ht]
    \centering
    \caption{Example \ref{Ex4}, The comparison of running times and relative errors of the proposed algorithm and the truncated t-SVD.}
    \label{tab:resnet_inferenc}
    \begin{tabular}{l c c}
        \multirow{2}{*}{} & \multicolumn{2}{c}{}\\\hline
         $\epsilon$ &  Proposed & Truncated t-SVD  \\\hline
        {$0.05$} & {\bf 9.18} Sec (Error=0.0497) & 120.81 Sec (Error=0.0286) \\
         {$0.07$} & {\bf 8.31} Sec (Error=0.0697) & 66.16 Sec (Error=0.0419)\\
        {$0.10$} & {\bf 4.265} Sec (Error=0.0990) & 28.08 Sec (Error=0.0616) 
        \\\hline
    \end{tabular}
\end{table*}

\section{Conclusion and future works}\label{Sec:Conclu}
In this paper, we proposed a new randomized fixed-precision algorithm for fast computation of tensor SVD (t-SVD). Unlike the existing randomized low tubal rank approximation methods, the proposed algorithm finds an optimal tubal rank and the corresponding low tubal rank approximation automatically given a data tensor and an approximation error bound. Simulations on synthetic and real-world data-sets confirmed that the proposed algorithm is efficient and applicable. This algorithm can be generalized to higher order tensors according to the paper \cite{martin2013order}. This will be our future work and we are planning to use it to develop fast tensor completion algorithms similar to the strategy done in \cite{che2022fast}. A detailed theoretical analysis of the proposed algorithm needs to be investigated. This is also our future work.

\section{Acknowledgements} 
After acceptance of the paper, the author found similar incremental algorithms for the computation of the t-SVD in \cite{ugwu2021tensor,ugwu2021viterative}. Thanks Dr. Ugochukwu Ugwu for bringing these works to the author's attention. The author would like to thank Stanislav Abukhovich for his help in the implementation of the algorithms. He introduced the author the \MATLAB function {\sffamily cat} to perform the tensor concatenation operation. He also pointed out constructive comments on the proof of Theorem 3. 
The author is also indebted to three anonymous reviewers for their constructive and
useful comments which have greatly improved the quality of the paper.  The author was partially supported by the Ministry of Education and Science of the
Russian Federation (grant 075.10.2021.068). \\



\ifCLASSOPTIONcaptionsoff
  \newpage
\fi
      \bibliography{BIBTEX_RAN}

\end{document}